\documentclass[a4paper,12pt,twoside,reqno]{smfart}
\usepackage[utf8x]{inputenc}
\usepackage[T1]{fontenc}
\usepackage[frenchb]{babel}
\usepackage{amssymb,amsmath}
\usepackage{smfthm}

\usepackage{stackrel}
\usepackage{pifont}
\usepackage{cases}

\usepackage{graphicx}

\usepackage{enumitem}
\usepackage{multicol}

\usepackage{amscd}
\usepackage[all,cmtip]{xy}

\theoremstyle{plain}
\newcounter{moncompteurtheoa}
\newenvironment{theo1}{\refstepcounter{moncompteurtheoa}\begin{enonce*}[plain]{Th\'eor\`eme \themoncompteurtheoa}}{\end{enonce*}}

\author{Hugues Bauch\`ere}
\address{
Laboratoire de math\'ematiques Nicolas Oresme, CNRS UMR 6139,
Universit\'e de Caen, Campus II, BP 5186,
14032 Caen Cedex, France}
\email{hugues.bauchere@unicaen.fr}
\urladdr{http://www.math.unicaen.fr/~bauchere/}

\title{Quelques remarques \`a propos d'un th\'eor\`eme de Checcoli}

\date{\today}

\DeclareMathOperator{\id}{id}

\DeclareMathOperator{\Gal}{Gal}

\newcommand{\N}{\mathbb{N}}

\newcommand{\F}{\mathbb{F}}
\newcommand{\Z}{\mathbb{Z}}

\newcommand{\Pcal}{\mathcal{P}}

\newcommand{\pcar}{\mathfrak{p}}
\newcommand{\qcar}{\mathfrak{q}}
\newcommand{\kbar}{\overline{k}}

\newcommand{\Mcal}{\mathcal{M}}
\newcommand{\cf}{\emph{cf.} }
\newcommand{\ie}{\emph{i.e.} }

\begin{document}

\raggedbottom

\frontmatter

\begin{abstract}
 Dans sa th\`ese (\cf \cite{SaraTh}), S. Checcoli montre, entre autres r\'esultats, que si $K$ est un corps de nombres
 et si $L/K$ est une extension galoisienne infinie de groupe de Galois $G$ d'exposant fini,
 alors les degr\'es locaux sur $L$ sont uniform\'ement born\'es en toutes les places de $K$.
 Dans cette article nous rassemblons deux remarques \`a propos de la g\'en\'eralisation
 du r\'esultat de S. Checcoli aux corps de fonctions de caract\'eristique positive.
 D'une part nous montrons un analogue de son th\'eor\`eme dans ce cadre,
 sous l'hypoth\`ese que l'exposant du groupe de Galois soit premier \`a $p$.
 D'autre part, nous montrons \`a l'aide d'un exemple que cette hypoth\`ese est en fait n\'ecessaire.
\end{abstract}

\alttitle{Some Remarks About a Checcoli Theorem}

\begin{altabstract}
 In his thesis (\cf \cite{SaraTh}), S. Checcoli shows that, among other results, if $K$ is a number field
 and if $L/K$ is an infinite Galois extension with Galois group $G$ of finite exponent,
 then $L$ has uniformly bounded local degrees at every prime of $K$.
 In this article we gather two remarks about the generalisation of S. Checcoli's result to function fields of positive characteristic.
 We first show an analogue of her theorem $2.2.2$ in this context,
 under the hypothesis that the Galois group exponent is prime to $p$.
 Using an example, we then show that this hypothesis is in fact necesary.
\end{altabstract}


\maketitle

\renewcommand{\contentsname}{Sommaire}
\tableofcontents

\section*{Introduction}

%

Soient $K$\label{notaK5} un corps global\footnote{\ie un corps de nombres ou une extension finie de $\F_q(T)$.}
et $L/K$
une extension galoisienne infinie.
On dit que l'extension $L/K$ a ses degr\'es locaux
\emph{uniform\'ement born\'es} en toutes
(resp. en presque toutes\footnote{\ie en toutes sauf un nombre fini.}) les places de~$K$
s'il existe un entier $d_0\in\N$ tel que pour toutes (resp. presque toutes) les places $v$ de $K$,
on a $[L_w:K_v]\leq d_0$ pour toute place $w|v$ de~$L$. 

Dans sa th\`ese S. Checcoli d\'emontre entre autres r\'esultats le th\'eor\`eme suivant (\cf th. $2.2.2$ de \cite{SaraTh}):

\enlargethispage{1\baselineskip}

\begin{theo1}\label{thdeCheccoli}
 Soit $K$ un corps de nombres et $L/K$ une extension galoisienne infinie.
 Alors les assertions suivantes sont \'equivalentes:
 \begin{enumerate}
  \item l'extension $L/K$ a ses degr\'es locaux uniform\'ement born\'es en toutes les places de $K$;
  \item l'extension $L/K$ a ses degr\'es locaux uniform\'ement born\'es en presque toutes les places de $K$;
  \item le groupe de Galois de l'extension $L/K$ est d'exposant fini.
 \end{enumerate}
\end{theo1}
Ensuite S. Checcoli et P. D\`ebes (\cf \cite{CHECCODEB}) ont g\'en\'eralis\'e ce r\'esultat \`a certaines classes de corps de fonctions.

Une question naturelle est donc de savoir si ce r\'esultat est encore vrai dans le cadre des corps de fonctions en caract\'eristique positive.
Nous montrons ici un analogue du th\'eor\`eme \ref{thdeCheccoli} en supposant de plus que l'exposant est premier \`a la caract\'eristique.
Cette hypoth\`ese n'est utile que dans la d\'emonstration de l'implication $3\Longrightarrow1$.
Nous donnons ensuite un exemple illustrant la n\'ecessit\'e de cette hypoth\`ese.

\mainmatter

\section{Analogue du th\'eor\`eme \ref{thdeCheccoli} en caract\'eristique positive}

Nous reprenons la preuve du th\'eor\`eme \ref{thdeCheccoli} en lui apportant les modifications n\'ecessaires dans notre cadre:
celui d'une extension finie $K$ de $k:=\F_q(T)$.
Nous commen\c{c}ons par rappeler la d\'efinition suivante.

\begin{defi}
 Soit $G$ un groupe. On dit que $G$ est un \emph{groupe m\'etacyclique}
 s'il existe un sous-groupe normal $H$ de $G$ tel que $H$ et~$G/H$ soient cycliques.
\end{defi}

\begin{rema}
 Si $G$ est un groupe fini m\'etacyclique d'exposant~$b$, alors $G$ est d'ordre un diviseur de $b^2$.
\end{rema}

Nous pouvons maintenant \'enoncer un analogue en caract\'eristique positive du th\'eor\`eme \ref{thdeCheccoli}.

\begin{theo}\label{thexpborne}
 Soient $K/k$ une extension finie et $L/K$ une extension galoisienne infinie. 
 Si les degr\'es locaux de l'extension $L/K$ sont
 uniform\'ement born\'es en presque toutes les places de $K$,
 alors $\Gal(L/K)$
 est d'exposant fini.

 R\'eciproquement, si $\Gal(L/K)$ est d'exposant fini premier \`a la caract\'eristique de $K$,
 alors les degr\'es locaux de l'extension $L/K$
 sont uniform\'ement born\'es en toutes les places de $K$.
\end{theo}

\begin{proof}
 Supposons que presque toutes les places de $K$ aient leurs degr\'es locaux sur $L$ born\'es par $d_0\in\N$.
 Soit $S$ l'ensemble des places de $K$ dont le degr\'e local sur $L$ n'est par born\'e par~$d_0$.
 Alors, par hypoth\`ese,~$S$ est un ensemble fini.
 Soient $E/K$ une extension galoisienne finie incluse dans~$L$ et $\sigma\in\Gal(E/K)$.
 D'apr\`es le th\'eor\`eme de densit\'e de Tchebotarev (\cf th. $9.13 A$ de \cite{ROSEN}),
 il existe une place finie $\pcar\in\Mcal_K\setminus S$ non ramifi\'ee sur $E$
 telle que $\sigma$ appartient \`a la classe de conjugaison d'Artin de~$\pcar$ dans $\Gal(E/K)$
 (rappelons que d'apr\`es le corollaire $3.5.5$ de \cite{STICHT},
 comme l'extension $E/K$ est s\'eparable finie, presque toutes les places de~$K$ sont non ramifi\'ees dans $E$).
 Ainsi, si $\qcar$ est une place de $E$ au-dessus de $\pcar$, il existe un conjugu\'e $\eta\in\Gal(E/K)$ de $\sigma$
 qui engendre le groupe de d\'ecomposition de $\qcar$ sur $\pcar$ qui est cyclique (d'apr\`es le th\'eor\`eme $3.8.2$ de \cite{STICHT})
 et isomorphe \`a $\Gal(E_\qcar/K_\pcar)$,
 o\`u $E_\qcar$ et $K_\pcar$ sont respectivement les compl\'et\'es des corps $E$ et $K$ en~$\qcar$ et~$\pcar$.
 Or, par hypoth\`ese, $\#\!\left(\Gal(E_\qcar/K_\pcar)\right)\leq d_0$, ainsi $\sigma^{d_0!}=\eta^{d_o!}=\id$
 et donc $\Gal(E/K)$ est d'exposant born\'e par~$d_0!$.
 Comme $\Gal(L/K)$ est la limite projective de la famille $\{Gal(E/K)\}_E$ index\'ee par les extensions galoisiennes finies de~$K$ contenues dans $L$,
 le groupe $\Gal(L/K)$ est d'exposant born\'e par~$d_0!$.

 R\'eciproquement, supposons que $\Gal(L/K)$ soit d'exposant fini $b\in\N$ premier \`a $p$, la caract\'eristique de $K$.
 \'Ecrivons $L$ comme une r\'eunion croissante d'extensions galoisiennes finies $L_j/K$ de groupe de Galois $G_j$.
 Soit $w$ une place de $L$, pour chaque $j$ on note $v_j$ l'unique place de $L_j$ en-dessous de~$w$ et~$L_{j,v_j}$ le compl\'et\'e de $L_j$ en la place $v_j$.
 De m\^eme, on note~$v$ l'unique place de $K$ en-dessous de $w$ et $K_v$ le compl\'et\'e de~$K$ en la place~$v$.
 On rappelle que pour tout $j$, le quotient du groupe de d\'ecomposition par le groupe d'inertie de $v_j$ sur $v$ est isomorphe
 au groupe de Galois des corps r\'esiduels et qu'il est donc cyclique (engendr\'e par le Frobenius).
 On rappelle \'egalement que le groupe $\Gal(L_{j,v_j}/K_v)$ est isomorphe au groupe de d\'ecomposition de $v_j$ sur $v$ qui est un sous-groupe de~$G_j$.
 Donc $\Gal(L_{j,v_j}/K_v)$ est d'exposant un diviseur de $b$. De plus, comme $b$ est premier \`a $p$, il ne peut y avoir de ramification sauvage.
 Il y a donc deux cas possibles pour l'extension $L_{j,v_j}/K_v$:
 \begin{enumerate}
  \item si l'extension $L_{j,v_j}/K_v$ est non ramifi\'ee, alors, $\Gal(L_{j,v_j}/K_v)$ est cyclique d'ordre un diviseur de~$b$;
  \item si l'extension $L_{j,v_j}/K_v$ est mod\'er\'ement ramifi\'ee, alors d'apr\`es la proposition $3.8.5$ de \cite{STICHT}, le groupe d'inertie de $v_j$ sur $v$ est cyclique,
   tout comme le quotient du groupe de d\'ecomposition par le groupe d'inertie, ainsi $\Gal(L_{j,v_j}/K_v)$ est m\'etacyclique
   et donc d'ordre un diviseur de~${b^2}$. 
 \end{enumerate}
 Pour tout $j$, le degr\'e de l'extension $L_{j,v_j}/K_v$ est donc un diviseur de $b^2$.
 On obtient donc le r\'esultat souhait\'e par passage \`a la limite projective. 
\end{proof}

\section{Un contre-exemple en caract\'eristique positive}  

Dans le th\'eor\`eme \ref{thexpborne}, on a vu que si $\Gal(L/K)$ est d'exposant fini premier \`a la caract\'eristique de~$K$,
alors les degr\'es locaux de l'extension $L/K$
sont uniform\'ement born\'es en toutes les places de~$K$.
Dans cette partie nous allons donner un contre-exemple dans le cas o\`u l'exposant du groupe $\Gal(L/K)$ est divisible par $p$.
Soit $\Pcal$\label{notaPcal} l'ensemble des polyn\^omes irr\'eductibles et unitaires de $\F_q[T]$.
Alors l'ensemble $\Pcal\cup\left\{\frac{1}{T}\right\}$
est en bijection avec l'ensemble $\Mcal_k$\label{notaMcalk} des places de $k:=\F_q(T)$.
Si $a\in\Pcal\cup\left\{\frac{1}{T}\right\}$\label{notaa5},
on note $v_a:k\longrightarrow\Z\cup\{\infty\}$\label{notava}
la valuation associ\'ee \`a $a$ et~$k_{v_a}$\label{notakva} le compl\'et\'e de $k$ en $v_a$.
Si $E$ est une extension finie de $k$ et si $w$ est une place de $E$ qui prolonge $v_a$,
on normalise (exceptionnellement ici)~$w$ par la formule $w(\alpha)=v_a(\alpha)$ pour tout $\alpha\in k$.
On a donc: 
\[w(E)=\frac{1}{e}\Z\cup\{\infty\}\text{,}\]
o\`u $e:=e(w/v_a)$ est l'indice de ramification de $w$ sur $v_a$.


\begin{lemm}\label{lemPajirrkv}
 Soient $a\in\Pcal\cup\left\{\frac{1}{T}\right\}$ et $i\in\N\setminus p\N$\label{notai5}. On pose:
 \[\label{notaPai}
  P_{a,i}(X):=X^p-X-a^{-i}\text{.}
 \]
 Soit $\theta_{a,i}\in\kbar$\label{notathetaai}
 une racine de $P_{a,i}$. On a les propri\'et\'es suivantes:
 \begin{enumerate}
  \item toutes les racines de $P_{a,i}$ sont de la forme $\theta_{a,i}+\zeta$ avec $\zeta\in\F_p$;
  \item $w(\theta_{a,i})=-\frac{i}{p}\notin\Z$ pour toute place $w|v_a$ de $k(\theta_{a,i})$;
  \item le polyn\^ome $P_{a,i}$ est irr\'eductible sur $k$;
  \item l'extension $k(\theta_{a,i})/k$ est totalement ramifi\'ee au-dessus de $a$;
  \item l'extension $k(\theta_{a,i})/k$ est cyclique de degr\'e $p$;
  \item $w(\theta_{a,i})=0$ pour tout ${b\in\Pcal\!\cup\!\left\{\frac{1}{T}\right\}\!\setminus\!\{a\}}$ et toute place $w|v_b$ de $k(\theta_{a,i})$;
  \item l'extension $k(\theta_{a,i})/k$ est non ramifi\'ee au-dessus de toutes les places de $k$ diff\'erentes de $a$.
 \end{enumerate}
\end{lemm}

\begin{proof}
 La propri\'et\'e $1$ est \'evidente. 
 V\'erifions la propri\'et\'e $2$. Soit $w|v_a$ une place de $k(\theta_{a,i})$. On a:
 \[w\!\left(\theta_{a,i}^p-\theta_{a,i}\right)=v_a\!\left(a^{-i}\right)=-i\]
 et
 \[w\!\left(\theta_{a,i}^p-\theta_{a,i}\right)\geq\min\{p\,w(\theta_{a,i}),w(\theta_{a,i})\}\text{,}\]
 ainsi $w(\theta_{a,i})<0$, donc la derni\`ere in\'egalit\'e est en fait une \'egalit\'e
 et on en d\'eduit que $w(\theta_{a,i})=-\frac{i}{p}$.
 Or $p\nmid i$ donc $w(\theta_{a,i})\notin\Z$ d'o\`u la propri\'et\'e $2$.
 
 Maintenant, on a $w(\theta_{a,i})=-\frac{i}{p}\in\frac{1}{e}\Z$, o\`u $e:=e(w/v_a)$ est l'indice de ramification de~$w$ sur~$v_a$.
 On en d\'eduit donc que $p|e$.
 Or, le polyn\^ome~$P_{a,i}$ \'etant de degr\'e $p$, l'extension $k(\theta_{a,i})/k$ est au plus de degr\'e $p$.
 Donc $e=p$.
 Les propri\'et\'es~$3$ \`a $5$ s'en d\'eduisent naturellement.
 Il ne nous reste donc plus qu'\`a montrer les propri\'et\'es $6$ et $7$.
 \[w\!\left(\theta_{a,i}^p-\theta_{a,i}\right)=v_b\!\left(a^{-i}\right)=0\]
 et
 \[w\!\left(\theta_{a,i}^p-\theta_{a,i}\right)\geq\min\{p\,w(\theta_{a,i}),w(\theta_{a,i})\}\text{,}\]
 ainsi $w(\theta_{a,i})=0$ et donc l'extension $k(\theta_{a,i})/k$ est non ramifi\'ee au-dessus de $b$.
\end{proof}

Dor\'enavant, pour all\'eger les notations, pour tous $a,b\in\Pcal\cup\left\{\frac{1}{T}\right\}$ et tout $i\in\N\setminus p\N$,
nous noterons encore $v_a$ une extension arbitraire de $v_a$ \`a~$k(\theta_{b,i})$.
Ainsi, nous aurons toujours:
\[v_{a}(\theta_{b,i})=
  \left\{\begin{array}{ll}
   -\frac{i}{p}\notin\Z & \text{si }b=a \\
   0 & \text{si } b\neq a
  \end{array}\right.\text{.}\]

Nous allons maintenant construire une $p$-extension ab\'elienne
\'el\'ementaire\footnote{\ie une extension dont le groupe de Galois est isomorphe \`a un produit de $\Z/p\Z$.}
dont le degr\'e local est infini au-dessus d'une place quelconque $a$ de $k$.
Pour cela, nous aurons besoin de montrer que certaines extensions sont lin\'eairement disjointes
(\cf \S$2.5$ de \cite{FRIEDJARDEN} pour ce qui concerne ces derni\`eres):

\begin{lemm}\label{propfamillealindisj}
 Soit $a\in\Pcal\cup\left\{\frac{1}{T}\right\}$.
 Alors, l'ensemble:
 \[\label{notaLambdaava}
  \Lambda_{a,v_a}:=\left\{k_{v_a}(\theta_{a,i})\,\big|\,i\in\N\setminus p\N\right\}
 \]
 forme une famille d'extensions lin\'eairement disjointes sur $k_{v_a}$.
\end{lemm}

\begin{proof}
 L'\'el\'ement $a$ de $\Pcal\cup\left\{\frac{1}{T}\right\}$ \'etant fix\'e, on simplifie les notations
 en posant $v:=v_a$ et $\theta_i:=\theta_{a,i}$ pour tout $i\in\N\setminus p\N$.

 Il faut montrer que toute sous-famille finie de $\Lambda_{a,v}$ est lin\'eairement disjointe sur $k_v$.
 Supposons par l'absurde qu'il existe une sous-famille finie $\left\{k_v(\theta_{i_j})\,\big|\,j\in\{1,\dots,n\}\right\}$
 de $\Lambda_{a,v}$ de cardinal $n\geq2$ qui ne soit pas lin\'eairement disjointe sur $k_v$.
 Sans perte de g\'en\'eralit\'e, on peut supposer que $n$ est minimal.
 L'extension $k_v(\theta_{i_n})/k_v$ \'etant galoisienne, on a:
 \begin{equation}\label{eqintersectiondeskvthetai}
  k_v(\theta_{i_n})\cap k_v(\theta_{i_1},\dots,\theta_{i_{n-1}})\neq k_v\text{.}
 \end{equation}
 En effet, si ce n'\'etait pas le cas 
 les extensions $k_v(\theta_{i_n})$ et $k_v(\theta_{i_1},\dots,\theta_{i_{n-1}})$ seraient lin\'eairement disjointes sur $k_v$
 (\cf remarque suivant le corollaire~$2.5.2$ de \cite{FRIEDJARDEN}),
 et donc aussi les extensions $k_v(\theta_{i_1}),\dots,k_v(\theta_{i_n})$, ce qui contredit le choix de $n$.
 Comme l'extension $k_v(\theta_{i_n})/k_v$ est de degr\'e~$p$ premier, on d\'eduit de (\ref{eqintersectiondeskvthetai}) que:
 \begin{equation}\label{eqraisareprd}
  \theta_{i_n}\in k_v(\theta_{i_1},\dots,\theta_{i_{n-1}})=k_v(\theta_{i_1},\dots,\theta_{i_{n-2}})(\theta_{i_{n-1}})\text{.}
 \end{equation}
 Il existe alors $\alpha_0,\dots,\alpha_{p-1}\in k_v(\theta_{i_1},\dots,\theta_{i_{n-2}})$ non tous nuls tels que:
 \[\theta_{i_{n}}=\alpha_0+\alpha_1\,\theta_{i_{n-1}}+\cdots+\alpha_{p-1}\,\theta_{i_{n-1}}^{p-1}\text{.}\]
 Or $\theta_{i_{n}}^p=\theta_{i_{n}}+a^{-i_{n}}$ et:
 \begin{multline*}
  \left(\alpha_0+\alpha_1\,\theta_{i_{n-1}}+\cdots+\alpha_{p-1}\,\theta_{i_{n-1}}^{p-1}\right)^p \\
   =\alpha_0^p+\alpha_1^p\left(\theta_{i_{n-1}}+a^{-i_{n-1}}\right)+\cdots+\alpha_{p-1}^p\,\left(\theta_{i_{n-1}}+a^{-i_{n-1}}\right)^{p-1}\text{.}
 \end{multline*}
 D'o\`u
 \begin{multline*}
  a^{-i_{n}}+\alpha_0+\alpha_1\,\theta_{i_{n-1}}+\cdots+\alpha_{p-1}\,\theta_{i_{n-1}}^{p-1} \\
   =\alpha_0^p+\alpha_1^p\left(\theta_{i_{n-1}}+a^{-i_{n-1}}\right)+\cdots+\alpha_{p-1}^p\,\left(\theta_{i_{n-1}}+a^{-i_{n-1}}\right)^{p-1}\text{.}
 \end{multline*}
 Comme les \'el\'ements $1,\theta_{i_{n-1}},\dots,\theta_{i_{n-1}}^{p-1}$ sont $k(\theta_{i_1},\dots,\theta_{i_{n-2}})$-lin\'eairement ind\'ependants,
 les coefficients des puissances de $\theta_{i_{n-1}}$ sont \'egaux dans l'\'egalit\'e pr\'ec\'edente.
 Ainsi, en comparant les coefficients des mon\^omes en~$\theta_{i_{n-1}}^{p-1}$, on obtient:
 \[\alpha_{p-1}=\alpha_{p-1}^p\text{,}\]
 d'o\`u on d\'eduit $\alpha_{p-1}\in\F_p$.
 En comparant les coefficients des mon\^omes en~$\theta_{i_{n-1}}^{p-2}$, on obtient:
 \begin{eqnarray*}
  \alpha_{p-2} & = & \alpha_{p-2}^p+\alpha_{p-1}^p\,(p-1)\,a^{-i_{n-1}} \\
  & = & \alpha_{p-2}^p-\alpha_{p-1}\,a^{-i_{n-1}}\text{.}
 \end{eqnarray*}
 Ainsi,
 \[\alpha_{p-2}^p-\alpha_{p-2}-\alpha_{p-1}\,a^{-i_{n-1}}=0\text{.}\]
 On a donc deux possibilit\'es: soit $\alpha_{p-1}=0$ et $\alpha_{p-2}\in\F_p$, soit $\alpha_{p-1}\neq0$
 et $\alpha_{p-2}=\alpha_{p-1}\,\theta_{i_{n-1}}+\zeta$ avec $\zeta\in\F_p$.
 Or, par minimalit\'e de $n$, on a ${\theta_{i_{n-1}}\notin k_v(\theta_{i_1},\dots,\theta_{i_{n-2}})}$,
 donc $\alpha_{p-1}=0$ et $\alpha_{p-2}\in\F_p$.
 De la m\^eme fa\c{c}on, on montre de proche en proche que $\alpha_{p-1}=\cdots=\alpha_2=0$ et $\alpha_1\in\F_p$.
 Il ne reste donc plus qu'\`a comparer les termes constants, \ie: 
 \[
  \alpha_0+a^{-i_{n}}=\alpha_0^p+\alpha_1^p\,a^{-i_{n-1}}
   =\alpha_0^p+\alpha_1\,a^{-i_{n-1}}\text{.}
 \]
 D'o\`u
 \[\alpha_0^p-\alpha_0-a^{-i_{n}}+\alpha_1\,a^{-i_{n-1}}=0\text{.}\]
 Deux cas s'offrent \`a nous:
 \begin{enumerate}
  \item $\alpha_1=0$ et $\alpha_0=\theta_{i_{n}}+\zeta$ avec $\zeta\in\F_p$;
  \item $\alpha_1\neq0$ et $\alpha_0=\theta_{i_{n}}-\alpha_1\,\theta_{i_{n-1}}+\zeta$ avec $\zeta\in\F_p$.
 \end{enumerate} 
 Dans le premier cas, on obtient:
 \[\theta_{i_{n}}\in k_v(\alpha_0)\subseteq k_v(\theta_{i_1},\dots,\theta_{i_{n-2}})\text{.}\]
 Donc les extensions $k_v(\theta_{i_1})$,\dots,$k_v(\theta_{i_{n-2}})$ et $k_v(\theta_{i_n})$
 ne sont pas lin\'eairement disjointes,
 ce qui contredit la minimalit\'e de $n$.
 Supposons que nous soyons dans le deuxi\`eme cas. Alors,
 il existe $\beta_{n-1}\in k_v(\theta_{i_1},\dots,\theta_{i_{n-2}})$ et $\lambda_{n-1}\in\F_p$ tels que:
 \[\theta_{i_{n}}=\beta_{n-1}+\lambda_{n-1}\,\theta_{i_{n-1}}\]
 (il suffit de poser $\beta_{n-1}:=\alpha_0-\zeta$ et $\lambda_{n-1}:=\alpha_1$).
 En reprenant le raisonnement pr\'ec\'edent et en rempla\c{c}ant pour tout $j\in\{1,\dots,n-2\}$
 la relation (\ref{eqraisareprd}) par:
 \[\theta_{i_n}\in k_v(\theta_{i_1},\dots,\theta_{i_{j-1}},\theta_{i_{j+1}},\dots,\theta_{i_{n-1}})(\theta_{i_j})\text{,}\]
 on montre que pour tout $j\in\{1,\dots,n-1\}$,
 il existe $\lambda_j\in\F_p$ \\
 et ${\beta_j\in k_v(\theta_{i_1},\dots,\theta_{i_{j-1}},\theta_{i_{j+1}},\dots,\theta_{i_{n-1}})}$ tels que:
 \[\theta_{i_{n}}=\beta_j+\lambda_j\,\theta_{i_j}\text{.}\]
 De plus, pour tout $j\in\{1,\dots,n-1\}$ on a:
 \begin{multline*}
  \theta_{i_{n}}-\lambda_1\,\theta_{i_1}-\cdots-\lambda_{n-1}\,\theta_{i_{n-1}} \\
   =\beta_j-\lambda_1\,\theta_{i_1}-\cdots-\lambda_{j-1}\,\theta_{i_{j-1}}-\lambda_{j+1}\,\theta_{i_{j+1}}-\cdots-\lambda_{n-1}\,\theta_{i_{n-1}}\text{,}
 \end{multline*}
 d'o\`u
 \[\theta_{i_{n}}-\lambda_1\,\theta_{i_1}-\cdots-\lambda_{n-1}\,\theta_{i_{n-1}}
  \in k_v(\theta_{i_1},\dots,\theta_{i_{j-1}},\theta_{i_{j+1}},\dots,\theta_{i_{n-1}})\text{,}\]
 et donc
 \[\theta_{i_{n}}-\lambda_1\,\theta_{i_1}-\cdots-\lambda_{n-1}\,\theta_{i_{n-1}}
  \in \bigcap_{1\leq j\leq n-1}k_v(\theta_{i_1},\dots,\theta_{i_{j-1}},\theta_{i_{j+1}},\dots,\theta_{i_{n-1}})\text{.}\]
 Or les corps $\left\{k_v(\theta_{i_j})\,\big|\,j\in\{1,\dots,n-1\}\right\}$ \'etant lin\'eairement disjoints
 (\`a nouveau par minimalit\'e de~$n$), on a:
 \[\bigcap_{1\leq j\leq n-1}k_v(\theta_{i_1},\dots,\theta_{i_{j-1}},\theta_{i_{j+1}},\dots,\theta_{i_{n-1}})=k_v\text{.}\]
 Ainsi, il existe $\beta\in k_v$ tel que:
 \[\theta_{i_{n}}=\beta+\lambda_1\,\theta_{i_1}+\cdots+\lambda_{n-1}\,\theta_{i_{n-1}}\text{.}\]
 On en d\'eduit que:
 \begin{equation}\label{eqetapelocal}
  v(\beta)=v(\theta_{i_{n}}-\lambda_1\,\theta_{i_1}-\cdots-\lambda_{n-1}\,\theta_{i_{n-1}})\text{.}
 \end{equation}
 Or $v(\theta_{i_j})=-\frac{i_j}{p}\notin\Z$ pour tout $j\in\{1,\dots,n\}$ d'apr\`es le lemme \ref{lemPajirrkv}
 et les entiers $i_j$ sont deux \`a deux distincts, d'o\`u:
 \[
  v(\theta_{i_{n}}-\lambda_1\,\theta_{i_1}-\cdots-\lambda_{n-1}\,\theta_{i_{n-1}})
   =\min_{1\leq j\leq n}\{v(\theta_{i_j})\} \\
   =-\frac{1}{p}\ \max_{1\leq j\leq n}\{i_j\}\notin\Z\text{.}
 \]
 Donc $v(\beta)\notin\Z$ ce qui contredit le fait que $\beta\in k_v$.
\end{proof}


\begin{rema}\label{propfamillealllindisj}
 \ \,Dans\ \,le\ \,lemme\ \,\ref{propfamillealindisj}\ \,nous\ \,n'avons\ \,fait\ \,varier\ \,que ${i\in\N\setminus p\N}$.
 Dans le cas global, on peut aussi faire varier $a\in\Pcal\cup\left\{\frac{1}{T}\right\}$.
 Plus pr\'ecis\'ement, on peut montrer que l'ensemble:
 \[\label{notaLambda5}
  \Lambda:=\left\{k(\theta_{a,i})\,\big|\,a\in\Pcal\cup\left\{\frac{1}{T}\right\},\,i\in\N\setminus p\N\right\}
 \]
 forme une famille d'extensions lin\'eairement disjointes sur $k$.
 La d\'emonstration est identique \`a celle du lemme \ref{propfamillealindisj}, \`a la diff\'erence qu'on consid\`ere
 une sous-famille finie $\left\{k(\theta_{a_j,i_j})\,\big|\,j\in\{1,\dots,n\}\right\}$
 de $\Lambda$ de cardinal $n\geq2$ qui ne soit pas lin\'eairement disjointe sur $k$,
 avec $n$ suppos\'e minimal. On montre alors (\cf \'egalit\'e (\ref{eqetapelocal})) qu'il existe $\beta\in k$ tel qu'on~ait:
 \[\theta_{a_n,i_{n}}=\beta+\lambda_1\,\theta_{a_1,i_1}+\cdots+\lambda_{n-1}\,\theta_{a_{n-1},i_{n-1}}\text{.}\]
 On en d\'eduit que:
 \[v_{a_n}(\beta)=v_{a_n}(\theta_{a_n,i_{n}}-\lambda_1\,\theta_{a_1,i_1}-\cdots-\lambda_{n-1}\,\theta_{a_{n-1},i_{n-1}})\text{.}\]
 Or, d'apr\`es le lemme \ref{lemPajirrkv}, pour tout $j\in\{1,\dots,n\}$, on a:
 \[v_{a_n}(\theta_{a_j,i_j})=
  \left\{\begin{array}{ll}
   -\frac{i_j}{p}\notin\Z & \text{si }a_j=a_n \\
   0 & \text{si } a_j\neq a_n
  \end{array}\right.\text{.}\]
 De plus, si $a_j=a_l=a_n$, alors $i_j\neq i_l$.
 Donc:
 \begin{eqnarray*}
  v_{a_n}(\theta_{a_n,i_{n}}-\lambda_1\,\theta_{a_1,i_1}-\cdots\!\!\!\! & - & \!\!\!\!\lambda_{n-1}\,\theta_{a_{n-1},i_{n-1}}) \\
   & = & \min_{1\leq j\leq n}\{v_{a_n}(\theta_{a_j,i_j})\} \\
   & = & -\frac{1}{p}\ \max\!\left\{i_j\,\big|\,j\in\{1,\dots,n\},\,a_j=a_n\right\}\notin\Z\text{.}
 \end{eqnarray*}
 Donc $v_{a_n}(\beta)\notin\Z$ ce qui contredit le fait que $\beta\in k$.
 L'ensemble $\Lambda$ forme donc bien une famille d'extensions lin\'eairement disjointes sur $k$.
\end{rema}

La proposition suivante nous donne un premier contre-exemple d'extension galoisienne infinie
d'exposant fini ayant une place dont tous les degr\'es locaux sont infinis.

\begin{prop}\label{coroalocalinf}
 Soient $a\in\Pcal\cup\left\{\frac{1}{T}\right\}$ et $L_a$\label{notaLa} le compositum des corps de la famille
 \[\Lambda_a:=\left\{k(\theta_{a,i})\,\big|\,i\in\N\setminus p\N\right\}\label{notaLambdaa}\text{.}\]
 Si $w$\label{notaw5} est une place de~$L_a$,
 on note $L_{a,w}$\label{notaLaw} le compl\'et\'e de $L_a$ en $w$.
 Alors, l'extension $L_a$ est une $p$-extension ab\'elienne de $k$ telle que
 pour toute place $w|v_a$ de $L_a$, l'extension $L_{a,w}/k_{v_a}$ est une $p$-extension ab\'elienne infinie.
\end{prop}

\begin{proof}
 Le fait que l'ensemble $\Lambda_a$ forme une famille d'extensions lin\'eairement disjointes de $k$
 et que pour toute place $w|v_a$, l'extension $L_{a,w}/k_{v_a}$ soit infinie d\'ecoule directement du lemme~\ref{propfamillealindisj}.
 De plus, les extensions $k(\theta_{a,i})/k$ \'etant ab\'eliennes de groupe de Galois isomorphe \`a $\Z/p\Z$,
 on en d\'eduit que l'extension $L_a$ est une $p$-extension ab\'elienne de~$k$.
\end{proof}

\begin{rema}
 Pour tout $b\in\Pcal\cup\left\{\frac{1}{T}\right\}\setminus\{a\}$ et toute place $w|v_b$ de~$L_a$,
 on a $\left[L_{a,w}:k_{v_b}\right]\leq p$.
 En effet, les extensions $k(\theta_{a,i})/k$ \'etant de degr\'e~$p$ et non ramifi\'ees au-dessus de $b$ d'apr\`es le lemme~\ref{lemPajirrkv},
 l'extension~$L_a/k$ est ab\'elienne, d'exposant $p$ et non ramifi\'ee au-dessus de $b$.
 Il en est donc de m\^eme des compl\'et\'es. Ainsi
 l'extension~$L_{a,w}$ est une $p$-extension ab\'elienne de $k_{v_b}$ qui est non ramifi\'ee au-dessus de $b$.
 Or les extensions non ramifi\'ees d'un corps local sont cycliques (\cf prop. p.~$113$ de \cite{FESVOS})
 ce qui n'est pas le cas de l'extension $L_{a,w}/k_{v_b}$
 d\`es que $\left[L_{a,w}:k_{v_b}\right]>p$ puisque son groupe de Galois est isomorphe \`a un produit de $\Z/p\Z$.
\end{rema}


Nous sommes maintenant en mesure de donner le contre-exemple annonc\'e:

\begin{theo}\label{coropextabinfdeglocinf}
 Soit $L$\label{notaL5} le compositum des corps de la famille $\Lambda$ de la remarque \ref{propfamillealllindisj}.
 Alors, $L$ est une $p$-extension ab\'elienne de~$k$ dont les degr\'es locaux sur $L$ sont infinis au-dessus de toutes les places de~$k$.
\end{theo}

\begin{proof}
 Comme $L$ est un compositum d'extensions de degr\'e~$p$, c'est clairement une $p$-extension ab\'elienne de $k$.
 Le fait que tous les degr\'es locaux sur $L$ soient infinis au-dessus des places de $k$ provient de la proposition \ref{coroalocalinf}.
 En effet, si $a\in\Pcal\cup\left\{\frac{1}{T}\right\}$, alors $L_a\subset L$.
\end{proof}


Nous terminons cette partie en donnant un exemple d'une extension galoisienne infinie de $k$
dont le groupe de Galois est d'exposant divisible par la caract\'eristique de $k$
et dont n\'eanmoins les degr\'es locaux sont uniform\'ement born\'es. 

\begin{prop}
 Soient $i\in\N\setminus p\N$ et $L_i$\label{notaLi} le compositum des corps de la famille
 \[\Lambda_i:=\left\{k(\theta_{a,i})\,\big|\,a\in\Pcal\cup\left\{\frac{1}{T}\right\}\right\}\label{notaLambdai}\text{.}\]
 Alors, $L_i$ est une $p$-extension ab\'elienne infinie de $k$ dont les degr\'es locaux sont uniform\'ement born\'es par $p^2$.
\end{prop}

\begin{proof}
 L'extension $L_i/k$ \'etant une sous-extension de l'extension $L/k$ (\cf th. \ref{coropextabinfdeglocinf}),
 c'est une $p$-extension ab\'elienne infinie de~$k$.
 Soient $a\in\Pcal\cup\left\{\frac{1}{T}\right\}$ et $L_{i\setminus a}$\label{notaLisansa}
 le compositum des corps de la famille
 $\Lambda_{i\setminus a}:=
  \left\{k(\theta_{b,i})\,\big|\,b\in\Pcal\cup\left\{\frac{1}{T}\right\}\setminus\{a\}\right\}$\label{notaLambdaisansa}.
 Alors $L_i$ est le compositum de $k(\theta_{a,i})$ et de $L_{i\setminus a}$.
 Soit $w|v_a$ une place de $L_i$, on note encore $w$ la restriction de $w$ \`a $L_{i\setminus a}$.
 L'extension $k(\theta_{a,i})/k$ \'etant totalement ramifi\'ee au-dessus de $a$, on a $\left[k_{v_a}(\theta_{a,i}):k_{v_a}\right]=p$.
 D'autre part, l'extension $L_{i\setminus a}/k$ est une $p$-extension ab\'elienne infinie qui est non ramifi\'ee au-dessus de $a$.
 Or les extensions non ramifi\'ees d'un corps local sont cycliques (\cf prop. p.~$113$ de \cite{FESVOS})
 ce qui n'est pas le cas de l'extension $L_{i\setminus a,w}/k_{v_a}$
 d\`es que $\left[L_{i\setminus a,w}:k_{v_a}\right]>p$.
 On obtient donc:
 \[\left[L_{i,w}:k_{v_a}\right]\leq\left[L_{i\setminus a,w}:k_{v_a}\right]\,\left[k_{v_a}(\theta_{a,i}):k_{v_a}\right]\leq p^2\text{.}\]
 D'o\`u le r\'esultat annonc\'e.
\end{proof}

\backmatter

\section*{Remerciements}

Je souhaite remercier Francesco Amoroso et Vincent Bosser mes directeurs de th\`ese ainsi que Bruno Angl\`es pour toutes les discussions que nous avons eus \`a propos de ce travail.

\bibliographystyle{smfalpha}
\bibliography{mabiblio}

\end{document}